\documentclass[12pt]{article}
\usepackage{amsmath,amsthm,amsfonts,amssymb,geometry}
\usepackage[]{xcolor}
\usepackage{mathtools,hyperref,enumerate}

\DeclareMathOperator{\CVP}{CVP}
\newcommand{\card}[1]{\left|#1\right|}
\newcommand{\nroot}{\zeta_n} 
\newcommand{\droot}{\zeta_d}
\newcommand{\bS}{\bar{S}}
    
\newtheorem{theorem}{Theorem}[section]
\newtheorem{proposition}[theorem]{Proposition}
\newtheorem{lemma}[theorem]{Lemma}
\newtheorem{corollary}[theorem]{Corollary}
\newtheorem{conjecture}[theorem]{Conjecture}

\newtheorem{problem}[theorem]{Problem}

\theoremstyle{definition}
\newtheorem{definition}{Definition}
\newtheorem{example}{Example}

\title{Necklaces, subset sums, and cyclic permutations}
\author{Robert Dougherty-Bliss \and Sergi Elizalde}
\date{\today}

\begin{document}

\maketitle

\begin{abstract}
    It is a well known that, for odd $n$, the number of subsets of $\{1,2,\dots,n\}$ the sum of whose elements is divisible by $n$ equals the number of binary necklaces of length $n$. In this paper generalize this result in two directions.

    On the one hand, we introduce a parameter $r$ so that requiring the subset sums to be congruent to $r$ modulo $n$ translates into imposing some periodicity conditions on the necklaces. 
    On the other hand, we refine these relations by the size $k$ of the subset, showing that it matches the number of ones in the necklace.
    We describe the precise conditions on $n$, $k$ and $r$ for which the equalities hold. We also extend some of our formulas to $q$-ary necklaces.

    The classical results correspond to the case $r=0$. When $r=1$, our identity is related to a conjecture of Baker et al.\ connecting subsets the sum of whose elements is congruent to $1$ modulo $n$ and unimodal permutations which consist of one cycle. We prove this conjecture using generating functions.

    Finding bijective proofs of most of our identities remains an open problem.
\end{abstract}

\section{Introduction}\label{sec:intro}

As shown by Stanley \cite[Exer.~1.105]{EC1}, the number of subsets of $[n]\coloneqq\{1,2,\dots,n\}$ the sum of whose elements is divisible by $n$ is equal to
\begin{equation}\label{eq:subsets_odd}
    \frac{1}{n} \sum_{\substack{d \mid n \\ d \text{ odd}}} \phi(d) 2^{n / d},
\end{equation}
where $\phi$ is Euler's totient function. It is well known that, when $n$ is
odd, this formula also counts (binary) necklaces of length $n$, which are
equivalence classes of binary words under rotation. Stanley posed the challenge
of finding a bijective proof of this fact, which was eventually done by Chan in~\cite{Chan}.

A similar formula, as well as another instance of a connection between subset sums and necklaces, appeared recently 
in work of Baker, Chen,
Li, and Qian~\cite{Baker}. They were studying the \emph{Gleason polynomials}
$G_n(c)$, recursively defined by the formula
\begin{equation*}
    \prod_{d \mid n} G_k(c) = \underbrace{(((c^2 + c)^2 + c)^2 + \cdots + c)^2 + c}_{n - 1 \text{ iterates}}.
\end{equation*}
Buff et al.~\cite{Buff} had proved that, for any positive integer $n$, the number of real roots of $G_n(c)$ is equal to the number of irreducible factors of $G_n(c)$ when considered as a polynomial over the two-element finite field, by showing that both sets are enumerated by
\begin{equation}\label{eq:necklaces_odd}
    \frac{1}{2n} \sum_{\substack{d \mid n \\ d \text{ odd}}} \mu(d) 2^{n / d},
\end{equation}
where $\mu$ is the M\"obius function. Baker et al. gave an explicit bijection between these sets, using several intermediate objects, all of which are counted by equation~\eqref{eq:necklaces_odd}.

One of such objects are primitive necklaces, which are necklaces all of whose rotations are different. It is known that equation~\eqref{eq:necklaces_odd} counts primitive necklaces of length $n$ with an odd number of ones. Another such object are V-shaped permutations of $[n]$ (i.e., permutations with a unique local minimum) which are cycles,
also known~\cite{ArcEli} to be enumerated by equation~\eqref{eq:necklaces_odd}.
Baker et al.~\cite{Baker} noticed that, according to~\cite[A000048]{oeis}, there is yet another object counted by equation~\eqref{eq:necklaces_odd}: subsets of $[n-1]$ whose elements sum to $1\bmod n$. 
They conjectured that the following refinement holds.

\begin{conjecture}[\cite{Baker}]\label{conj:Baker} The number of cyclic V-shaped permutations of $[n]$ whose minimum is at position $k$ equals the number of $(k-1)$-element subsets of $[n-1]$ whose elements sum to $1\bmod n$.
\end{conjecture}

Our first main result (Theorem~\ref{thm:S=CVP}) is a proof of this conjecture. We also show, by applying a bijection from~\cite{ArcEli,Ste}, that the connection that was observed in~\cite{Baker} between primitive necklaces and subsets that sum to $1\bmod n$ admits a natural refinement where the number of ones in the necklace corresponds to the size of the subset. The analogous refinement also applies to the equality from \cite[Exer.~1.105]{EC1} involving arbitrary necklaces
and subsets whose sum is divisible by~$n$.

In fact, as we will see in Theorem~\ref{thm:odd_nk}, both of these results are special cases of a
more general framework relating necklaces whose period satisfies certain
divisibility conditions, and subsets whose sum is congruent to a fixed value
modulo~$n$.

The remainder of this paper is structured as follows. In Section~\ref{sec:definitions} we give the necessary background on necklaces, permutations, and subset sums, and we introduce some notation. We also describe, in Section~\ref{sec:r=0}, the classical connection between subsets of $[n]$ whose sum is divisible by $n$ and necklaces of length~$n$, when $n$ is odd, as well as a recent refinement by the size of the subset, which can be viewed as inspiration for our work.
In Section~\ref{sec:r=1} we prove Conjecture~\ref{conj:Baker}, and then we use known bijections between permutations and necklaces to deduce a connection between subsets of $[n]$ whose sum is equal to $1\bmod n$ and primitive necklaces of length~$n$.

In Section~\ref{sec:general}, we put the two connections described above under the same umbrella, proving a more general result that relates subsets of $[n]$ whose sum is equal to $1\bmod n$ and necklaces whose co-period (defined in Section~\ref{sec:necklaces}) divides~$r$. As in the previous cases, the size of the subset $k$ corresponds to the number of ones in the necklace. In Theorem~\ref{thm:Nr=S'r}, we describe the precise conditions on $n$, $r$, and $k$ for which the equality holds.

In Section~\ref{sec:q} we extend some of our results to $q$-ary necklaces and multisets, refining and generalizing an identity found by Chan~\cite{Chan}.
Finally, in Section~\ref{sec:open}, we list several open questions that our work raises, mostly about finding bijective proofs of several equalities relating necklaces, subset sums, and permutations.

\section{Definitions and background}
\label{sec:definitions}

\subsection{Necklaces}\label{sec:necklaces}

Two binary words $u,v$ of length $n$ are conjugate if they are cyclic rotations of each other, that is, there exist words $r$ and $s$ such that $u=rs$ and $v=sr$. A {\em binary necklace} of length $n$ is a conjugacy class of binary words of length $n$; we write $(u)$ to denote the conjugacy class of the word $u$. Throughout the paper, we will use the term {\em necklace} to refer to a binary necklace, i.e., over the alphabet $\{0,1\}$.  A word $u$ is {\em primitive} if it is not the power of another word, i.e., it is not of the form $u=v^j$ for $j\ge2$. If $u=v^j$ and $v$ is primitive, we call $j$ the {\em co-period} of $u$.
A necklace is {\em primitive} if it is the conjugacy class of a primitive word.

A {\em Lyndon word} is a primitive word that is lexicographically smaller than all the other words in its conjugacy class. Lyndon words are in one-to-one correspondence with primitive necklaces, since each conjugacy class of primitive words has a unique lexicographically smallest element.

\begin{definition}
    For any integers $0\le k\le n$ and $0\le r< n$, define the following sets:
    \begin{itemize}     
        \item $N(n, k)$ is the set of binary necklaces of length $n$ with $k$ ones. 
        \item $L(n,k)$ is the set of necklaces in $N(n,k)$ that are primitive; equivalently, the set of binary Lyndon words of length $n$ with $k$ ones.
        \item $N_r(n, k)$ is the set of necklaces in $N(n,k)$ whose co-period divides $r$. In particular, $N_0(n, k)=N(n, k)$ and
        $N_1(n, k)=L(n, k)$.
    \end{itemize}
\end{definition}

If $A(n,k)$ is any two-parameter family of sets, define 
\begin{equation}\label{eq:A}
A(n)=\bigcup_{k=0}^n A(n,k).
\end{equation}

The enumeration of necklaces is a well-studied problem in combinatorics. Here are some classical formulas for binary necklaces.

\begin{lemma}\label{lem:NL} 
For any $0\le k\le n$, we have
\begin{align*}\card{N(n)}=\frac{1}{n}\sum_{d\mid n}\phi(d) 2^{n/d}, &&
\card{N(n,k)}=\frac{1}{n}\sum_{d\mid \gcd(n,k)}\phi(d) \binom{n/d}{k/d},\\
\card{L(n)}=\frac{1}{n}\sum_{d\mid n}\mu(d) 2^{n/d}, &&
\card{L(n,k)}=\frac{1}{n}\sum_{d\mid \gcd(n,k)}\mu(d) \binom{n/d}{k/d}.
\end{align*}
\end{lemma}

\begin{proof}
    The formulas for $N(n)$ and $N(n,k)$ are classical applications of Burnside's lemma, see e.g.~\cite{StanleyAC}.
    
    The formulas for $L(n)$ and $L(n,k)$ are obtained using M\"obius inversion~\cite{Lothaire}. For example, to prove the last equality, let $f(n,k)$ be the number of binary words of length $n$ with $k$ ones, and let $g(n,k)$ be the number of such words that are primitive. Clearly, $f(n,k)=\binom{n}{k}$. Since any word can be written uniquely as a power of a primitive word, we have $f(n,k)=\sum_{d\mid\gcd(n,k)}g(n/d,k/d)$. Thus, by M\"obius inversion,
    $$g(n,k)=\sum_{d\mid\gcd(n,k)}\mu(d)f(n/d,k/d)=\sum_{d\mid\gcd(n,k)}\mu(d)\binom{n/d}{k/d}.$$
    Finally, use that $\card{L(n,k)}=g(n,k)/n$.
\end{proof}

\subsection{Permutations}

A permutation $\sigma$ of $[n]=\{1,2,\dots,n\}$ is called {\em V-shaped} if $$\sigma_1>\sigma_2>\dots>\sigma_k<\sigma_{k+1}<\dots<\sigma_n$$ for some $k\in[n]$. Even though these permutations are called {\em unimodal} in~\cite{Baker}, we prefer to reserve the term unimodal for permutations such that $$\sigma_1<\sigma_2<\dots<\sigma_k>\sigma_{k+1}>\dots>\sigma_n$$ for some $k\in[n]$, following~\cite{Thibon,ArcEli}. Note that the reverse-complement operation $\sigma_1\dots\sigma_n\mapsto (n+1-\sigma_n)\dots(n+1-\sigma_1)$ takes unimodal permutations to V-shaped permutations, while preserving the cycle structure.

A cyclic permutation is a permutation that consists of a single cycle. There are $(n-1)!$ cyclic permutations of $[n]$.

\begin{definition}
    For any $1\le k\le n$, let $\CVP(n, k)$ be the set of cyclic V-shaped permutations of $[n]$ that have their minimum in position $k$; we call these {\em cyclic V-shaped permutations}.
\end{definition}

\subsection{Subset sums}
\begin{definition}
    For any integers $0\le k\le n$ and $0\le r< n$, define the following sets:
    \begin{itemize}
    \item $S_r(n, k)$ is the set of $k$-element subsets of $[n - 1]$ whose elements sum to $r \bmod n$.
    \item $\bS_r(n, k)$ is the set of $k$-element subsets of $[n]$ whose elements sum to $r \bmod n$.
    \end{itemize}
Define also $S_r(n)$ and $\bS_r(n)$ as in equation~\eqref{eq:A}.
\end{definition}

By identifying $n$ and $0$, one could equivalently define $\bS_r(n, k)$ as the set of $k$-element subsets of $\{0,1,\dots,n-1\}$ whose elements sum to $r \bmod n$.

For any $1\le k\le n$, there is a trivial bijection between $\bS_r(n,k)$ and $S_r(n,k)\cup S_r(n,k-1)$, obtained by separating the sets in $\bS_r(n,k)$ into those that do not contain $n$ and those that do. It follows that
\begin{equation}\label{eq:SS'}
\card{\bS_r(n,k)}=\card{S_r(n,k)}+\card{S_r(n,k-1)}.
\end{equation}

\subsection{The $r=0$ case: arbitrary necklaces}
\label{sec:r=0}

Subsets of $[n]$ whose elements sum to $0\bmod n$ were considered in Exercise 105a in Chapter 1 of~\cite{EC1}. Stanley shows that
\begin{equation}    \label{eq:S0'}
    \card{\bS_0(n)}=\frac{1}{n} \sum_{\substack{d \mid n \\ d \text{ odd}}} \phi(d) 2^{n / d}
\end{equation}
for all $n$.
When $n$ is odd, this coincides with the formula for $\card{N(n)}$ given in
Lemma~\ref{lem:NL}. Part (b) of the same exercise poses the problem of finding a
bijective proof of this result. A bijection between $\bS_0(n)$ and $N(n)$ for any
odd $n$ was eventually found by Chan~\cite{Chan}, along with generalizations to
$q$-ary necklaces and multisets of $[n]$ where the multiplicity of each element
is less than~$q$.

It was noted in~\cite{Chan} that the equality $\card{\bS_0(n)}=\card{N(n)}$ for odd $n$
admits a natural refinement, namely $\card{\bS_0(n,k)}=\card{N(n,k)}$ for any $0\le k\le
n$. Chan's bijection does not provide a bijective proof of this. In Theorem~\ref{thm:Nr=S'r} we will
characterize all values of $k$ and $n$ for which these cardinalities coincide.

\begin{problem}\label{prob:Chan}
    Find a bijection between $\bS_0(n,k)$ and $N(n,k)$ for any odd $n$ and any $0\le k\le n$. 
\end{problem}

\section{The $r=1$ case: primitive necklaces}
\label{sec:r=1}

\subsection{A conjecture about subset sums and permutations}

In their study of the roots of Gleason polynomials, Baker et al.~\cite{Baker} found an intriguing connection between subsets of $[n-1]$ whose elements sum to $1\bmod n$ and V-shaped permutations on $[n]$. They noticed that both sets have the same cardinality, and they conjectured that the equality could be refined by keeping track of the position of the minimum in the permutation on the one hand, and the size of the subset on the other hand. In this section we will prove this conjecture, stated above as Conjecture~\ref{conj:Baker}.

\begin{theorem}\label{thm:S=CVP} For any $1\le k\le n$, we have
    $$\card{S_1(n, k-1)} = \card{\CVP(n, k)}.$$
In particular, summing over $k$, we have $\card{S_1(n)}=\card{\CVP(n)}$.
\end{theorem}

For the right-hand side, we use the following generating polynomial for V-shaped permutations with respect to the position of their minimum, which can be deduced from Thibon's work~\cite{Thibon}.

\begin{lemma}[\cite{Thibon}] \label{lem:thibon}
We have
\[
    \sum_{k=1}^n\card{\CVP(n,k)}x^{k-1}=\frac{1}{n(1+x)}\sum_{d\mid n}\mu(d)\left(1-(-x)^d\right)^{n/d}.
\]
\end{lemma}

Next we find a generating polynomial for the left-hand side of
Theorem~\ref{thm:S=CVP}.

\begin{lemma}
    \label{lem:S1}
    We have
$$
 \sum_{k=1}^n\card{S_1(n,k-1)}x^{k-1}=\frac{1}{n(1+x)} \sum_{d \mid n} \mu(d) \left(1 - (-x)^d\right)^{n/d}.
  $$  
\end{lemma}

\begin{proof}
Consider the generating function
\[    
    f(x,t) = \prod_{j = 1}^{n - 1} (1 + xt^j),
\]
which enumerates the subsets of $[n - 1]$ with respect to their size and their sum. To extract the coefficients that correspond to sums that are $1 \bmod n$, let $\nroot$ be a primitive $n$th root of unity. We have
\begin{equation}
    \label{eq:fourier}
   \sum_{k=1}^n\card{S_1(n,k-1)}x^{k-1}=
    \frac{1}{n} \sum_{\ell = 0}^{n - 1} \nroot^{-\ell} f(x,\nroot^\ell).
\end{equation}

    Fix $\ell$ and let $\omega\coloneqq\nroot^\ell$, which is a root
    of unity of order $m\coloneqq n / {\gcd(n, \ell)}$. Write 
    \begin{equation*}
        f(x,\nroot^\ell) = f(x,\omega)= \prod_{j = 1}^{n - 1} (1 + x \omega^j),
    \end{equation*}
    and note that the factors only depend on the congruence class of $j$ modulo $m$. For each $1\le s\le m-1$, the factor $1 + x \omega^s$ appears $n/m$ times, whereas the factor $1 + x \omega^m=1 + x$ appears $n/m-1$ times. Thus, we have
    \begin{equation*}
        f(x,\nroot^\ell)
        =
        \frac{1}{1 + x} \left(\prod_{s = 0}^{m-1} (1 + x \omega^s)\right)^{n/m}.
    \end{equation*}
    
    Dividing the expression $\prod_{s=0}^{m-1}(z-\omega^s)=z^m-1$ by $z^m$ and substituting $z=(-x)^{-1}$, we get 
    \begin{equation*}
        \prod_{s = 0}^{m-1} (1 + x \omega^s)
        =
        1 - (-x)^m.
    \end{equation*}
    Therefore,
    \begin{equation*}
        \nroot^{-\ell} f(x,\nroot^\ell)
        =
        \frac{\nroot^{-\ell}}{1+x} \left(1 - (-x)^{n / {\gcd(n,\ell)}}\right)^{\gcd(n,\ell)}.
    \end{equation*}
    
    Plugging this back into  equation \eqref{eq:fourier} and grouping the terms
    of the sum where $\gcd(n,\ell)=d$, for each divisor $d$ of $n$, we obtain
    \begin{align*}
        \sum_{k=1}^n\card{S_1(n,k-1)}x^{k-1} 
        & =\frac{1}{n} \sum_{\ell=0}^{n-1} \frac{\nroot^{-\ell}}{1 + x} \left(1 - (-x)^{n / \gcd(n,\ell)}\right)^{\gcd(n,\ell)}\\
        &=         \frac{1}{n(1 + x)} \sum_{d \mid n} \left(1 - (-x)^{n / d}\right)^d \left(\sum_{\substack{\ell=0\\ \gcd(n, \ell) = d}}^{n-1} \nroot^{-\ell} \right).
    \end{align*}
    It is well known that the sum of all the roots of unity of order $n / d$ is equal to $\mu(n/d)$, and so
    \begin{align*}
       \sum_{k=1}^n\card{S_1(n,k-1)}x^{k-1}
        = \frac{1}{n(1 + x)} \sum_{d \mid n} \left(1 - (-x)^{n / d}\right)^d \mu(n/d),
    \end{align*}
    which is equivalent to the stated formula.
\end{proof}

\begin{proof}[Proof of Theorem~\ref{thm:S=CVP}]
The expressions for the generating functions in Lemmas~\ref{lem:thibon} and~\ref{lem:S1} are equal, so the coefficients must satisfy $\card{S_1(n,k-1)}=\card{\CVP(n,k)}$.
\end{proof}

\subsection{Bijections between permutations and necklaces}

Finding a bijective proof of Theorem~\ref{thm:S=CVP} remains an open question. In this subsection we describe a bijection between V-shaped permutations and primitive necklaces. This will allow us to use Theorem~\ref{thm:S=CVP} to deduce a statement relating subset sums and primitive necklaces, which can be viewed as an analogue of the equality $\card{\bS_0(n,k)}=\card{N(n,k)}$ that we discussed in Section~\ref{sec:r=0}.

Permutations with a given cycle structure and descent set were studied in a seminal paper by Gessel and Reutenauer~\cite{GR}. They showed that permutations whose descent set is contained in a given set (equivalently, permutations forced to have ascents in given positions) are in bijection with multisets of primitive necklaces, and this bijection preserves the cycle structure.

Their construction was later generalized by Steinhardt~\cite{Ste} and by Gessel, Restivo and Reutenauer~\cite{GRR} to permutations that are forced to have ascents and descents in certain positions, although the resulting necklaces are not always primitive. A similar construction was used by Archer and Elizalde~\cite{ArcEli} in their study of cyclic permutations realized by the tent map, in order to give a bijection between primitive binary necklaces with an odd number of ones and cyclic permutations that are unimodal (or equivalently, by applying the reverse-complement operation, cyclic permutations that are V-shaped).

Let us rephrase this bijection in our setting. First, define the set
\begin{equation}\label{eq:L+} 
L^+(n,k)=\begin{cases} L(n,k)\cup L(n/2,k/2) & \text{if $n$ is even and $k\equiv 2\bmod 4$},\\ L(n,k) & \text{otherwise.}\end{cases}
\end{equation}
We will use the convention $\CVP(n,0)=\CVP(n,n+1)=\emptyset$.

\begin{theorem}[\cite{Ste,ArcEli}]\label{thm:ArcEli}
    For any $0\le k\le n$, there is a bijection 
    \begin{equation}\label{eq:AEbij} \Psi:\CVP(n,k)\cup\CVP(n,k+1)\to L^+(n,k)\end{equation} 
\end{theorem}

We can describe $\Psi$ as follows.  Let $\sigma\in \CVP(n,k)\cup\CVP(n,k+1)$, that is, $\sigma$ is a cyclic permutation that satisfies $\sigma_1>\sigma_2>\dots>\sigma_k$ and $\sigma_{k+1}<\sigma_{k+2}<\dots<\sigma_n$.
To compute $\Psi(\sigma)$, write $\sigma$ in cycle notation, say $\sigma=(a_1,a_2,\dots,a_n)$, and replace each $a_i$ with a $1$ if $a_i\le k$, and with a $0$ if $a_i>k$. This yields a binary necklace $(w)$ with $k$ ones.

It follows from \cite[Thm.~2.2]{Ste} that $(w)$ is either primitive or of the form $(vv)$, where $v$ is primitive with an odd number, $k/2$, of ones. 
In the first case, we let $\Psi(\sigma)=(w)$; in the second
case, which can only happen when $n$ is even and $k\equiv2\bmod4$, we let $\Psi(\sigma)=(v)$. In both cases, $\Psi(\sigma)\in L^+(n,k)$.

To describe $\Psi^{-1}$, define first a map $s$ that takes a binary word $w$ to the binary word $s=s(w)$ of the same length, whose entries are the partial sums of $w$, that is, $s_j=w_1+\dots+w_j\bmod 2$ for all $j\in[n]$.

Given $(w)\in L(n,k)$, we can compute $\Psi^{-1}(w)$ as follows. Write $w=w_1\dots w_n$, and consider the cyclic shifts $w^{(j)}=w_jw_{j+1}\dots w_nw_1w_2\dots w_{j-1}$ for $1\le j\le n$. Order the words $s(w^{(j)})$ lexicographically from largest to smallest, and let $a_j$ be the position of the word $s(w^{(j)})$ in this ordering. Finally, let $\Psi^{-1}(w)=\sigma=(a_1,a_2,\dots,a_n)$. It follows from \cite{Ste,ArcEli} that $\sigma_1>\sigma_2>\dots>\sigma_{k}$
and $\sigma_{k+1}<\sigma_{k+2}<\dots<\sigma_n$, that is, $\sigma\in \CVP(n,k)\cup\CVP(n,k+1)$. 

Now suppose that $n$ is even and $k\equiv 2\bmod 4$, and let $(v)\in L(n/2,k/2)$. To compute $\Psi^{-1}(v)$, let $w=vv$ and apply the construction from the previous paragraph, with the following rule to break the ties caused by $s(w^{(j)})=s(w^{(j+n/2)})$: the order relationship between $a_{j+1}$ and $a_{j+1+n/2}$ is the same as the order relationship between $a_{j}$ and $a_{j+n/2}$ if $w_j=0$, otherwise it is reversed.

\begin{example}
    Let $n=5$ and $k=3$, and let $\sigma=54213=(1,5,3,2,4)\in\CVP(5,4)$. Replacing each element with a $1$ or a $0$ depending on whether it is at most $3$ or not, we get the necklace $(1,0,1,1,0)\in L(5,3)$.

    Conversely, given $w=10110$, its cyclic shifts $w^{(j)}$ for $1\le j\le 5$ are 
    $$10110,01101,11010,10101,01011,$$
    and their partial sums are 
    $$11011,01001,10010,11001,01101.$$
    Ordering them lexicographically from largest to smallest recovers the permutation $\sigma=(1,5,3,2,4)=54213\in\CVP(5,4)$.
\end{example}

\begin{example}
    Let $n=6$ and $k=2$, and let $\sigma=651234=(1,6,4,2,5,3)\in\CVP(6,3)$. Replacing each element with a $1$ or a $0$ depending on whether it is at most $2$ or not, we get the necklace $(1,0,0,1,0,0)$, which is of the form $(vv)$ for $(v)=(1,0,0)\in L(3,1)$.

    Conversely, from $v=100$, take $w=vv=100100$ and consider its cyclic shifts
    $$100100,001001,010010,100100,001001,010010,$$
    and their partial sums 
    $$111000,001110,011100,111000,001110,011100.$$
    Ordering them lexicographically from largest to smallest, breaking ties with the above convention, recovers the permutation $\sigma=(1,6,4,2,5,3)=651234\in\CVP(6,3)$.
\end{example}

Theorem~\ref{thm:ArcEli} and Lemma~\ref{lem:NL} imply the following enumeration formula for cyclic V-shaped permutations whose minimum is in one of two positions.

\begin{corollary}\label{cor:CVP-formula}
    For any $0\le k\le n$, 
\begin{multline*}
    \card{\CVP(n,k)}+\card{\CVP(n,k+1)}=\card{L^+(n,k)}\\
    =\begin{cases}
    \displaystyle
    \frac{1}{n}\sum_{d\mid\gcd(n,k)}\mu(d)\binom{n/d}{k/d}+\frac{2}{n}\sum_{d\mid\gcd(\frac{n}{2},\frac{k}{2})}\mu(d)\binom{\frac{n}{2d}}{\frac{k}{2d}} & \text{if $n$ is even and $k\equiv 2\bmod 4$},\\ 
    \displaystyle
    \frac{1}{n}\sum_{d\mid\gcd(n,k)}\mu(d)\binom{n/d}{k/d} & \text{otherwise.}\end{cases}
\end{multline*}
\end{corollary}

Using inclusion-exclusion and some algebraic manipulations, it follows from Corollary~\ref{cor:CVP-formula} that
$$\card{\CVP(n,k)}=\sum_{i=0}^{k-1}(-1)^{k-1-i}\card{L^+(n,i)}=\frac{1}{n}\sum_{d\mid n}\mu(d)\sum_{j=0}^{\lfloor\frac{k-1}{d}\rfloor}(-1)^{k+j-1}\binom{n/d}{j}.
$$
This is the same expression for $\card{\CVP(n,k)}$ that is obtained when extracting the coefficient of $x^{k-1}$ in Lemma~\ref{lem:thibon}.

We can now combine Theorems~\ref{thm:S=CVP} and~\ref{thm:ArcEli} to obtain a relationship between subsets that sum to $1 \bmod n$ and primitive necklaces. In contrast with the well-known relationship described in Section~\ref{sec:r=0} between subsets that sum to $0 \bmod n$ and arbitrary necklaces, this appears to be a new identity, even without the refinement by $k$.

\begin{corollary}\label{cor:S'1=L}
    For any $n\ge2$ and $0\le k\le n$, we have
    $$\card{\bS_1(n, k)} = \card{L^+(n, k)}.$$
    In particular, $\card{\bS_1(n, k)} = \card{L(n, k)}$ when $n$ or $k$ are odd, and $\card{\bS_1(n)} = \card{L(n)}$ when $n$ is odd.
\end{corollary}

\begin{proof}
The statement is trivial for $k=0$. For $1\le k\le n$, we have
\begin{align*}
    \card{\bS_1(n, k)}&=\card{S_1(n,k-1)}+\card{S_1(n,k)} && \text{(by Equation~\eqref{eq:SS'})}\\
    &=\card{\CVP(n,k)}+\card{\CVP(n,k+1)}  && \text{(by Theorem~\ref{thm:S=CVP})}\\
    &=\card{L^+(n,k)} && \text{(by Theorem~\ref{thm:ArcEli}).}
\end{align*}

When $n$ or $k$ are odd, we use the fact that $L^+(n,k)=L(n,k)$ by definition. When $n$ is odd, summing over $k$ gives $\card{\bS_1(n)} = \card{L(n)}$.
\end{proof}

Figure~\ref{fig:diagram} summarizes the relationships proved in this section.

\begin{figure}[h]
$$\begin{array}{cclc}
\CVP(n,k)\cup\CVP(n,k+1) & \overunderset{\text{bijective}}{\text{Thm.~\ref{thm:ArcEli}}}{\longleftrightarrow} & L^+(n,k)\\
\updownarrow \text{\scriptsize non-bijective, Thm.~\ref{thm:S=CVP}} \updownarrow & &\\
S_1(n,k-1)\cup S_1(n,k) & \overunderset{\text{bijective}}{\text{Eq.~\eqref{eq:SS'}}}{\longleftrightarrow} & \bS_1(n,k)
\end{array}$$
\caption{A diagram of the relationships described in Section~\ref{sec:r=1}, for $1\le k\le n-1$.}
\label{fig:diagram}
\end{figure}

We conclude this section with two immediate consequences of Theorem~\ref{thm:ArcEli} obtained by expressing $\CVP(n)$ as a union of sets $\CVP(n,k)\cup\CVP(n,k+1)$ over all odd or over all even values of $k$, respectively.

\begin{corollary}\label{cor:CVPbijections}
For any~$n$, the map $\Psi$ gives bijections
$$\CVP(n)\to \bigcup_{k \text{ odd}}  L(n,k) \quad\text{and}\quad \CVP(n)\to\bigcup_{k \text{ even}} L^+(n,k).$$
\end{corollary}

The first union (over odd $k$) is the set of primitive necklaces of length $n$ with an odd number of ones.
By equation~\eqref{eq:L+}, the second union (over even $k$) is the set of primitive necklaces of length $n$ with an even number of ones, together, if $n$ is even, with the set of primitive necklaces of length $n/2$ with an odd number of ones.
Thus, Corollary~\ref{cor:CVPbijections} recovers the bijections described in~\cite{Baker} between $\CVP(n)$ and these two sets, denoted in~\cite{Baker} by $N^-(n)$ and $\tilde{N}^+(n)$, respectively.

\section{The general framework} \label{sec:general}

\subsection{Formulas for arbitrary $r$}

In Section~\ref{sec:r=0} we discussed a well-known relation between subsets that sum to $0\bmod n$ and arbitrary necklaces. In Section~\ref{sec:r=1} we considered a seemingly new relation between subsets that sum to $1\bmod n$ and primitive necklaces. 
In this section we show that these relations are instances of a more general framework that relates subsets of $[n]$ that sum to $r \bmod n$ and binary necklaces whose co-periods (as defined in Section~\ref{sec:necklaces}) divide $r$. 
When $r = 0$, these are necklaces with no restrictions (Section~\ref{sec:r=0}), and when $r = 1$, these are primitive necklaces (Section~\ref{sec:r=1}).

To compare the quantities $\card{\bS_r(n, k)}$ and $\card{N_r(n, k)}$, we will first show that they can both be expressed in a similar form. These expressions involve the following sums of powers of a primitive $d$-root of unity $\droot$, often called {\em Ramanujan sums} (see \cite[Ch.~8.3]{apostol}, \cite{apostol-paper}, or \cite[A086831]{oeis}):
\begin{equation}\label{eq:Ramanujan}
        c_d(r) \coloneqq \sum_{\substack{j = 1 \\ \gcd(j, d) = 1}}^d \droot^{j r} = \sum_{j \mid \gcd(r, d)} j \mu(d / j) 
               = \mu(d / \gcd(r, d))  \frac{\phi(d)}{\phi(d / \gcd(r, d))}.
\end{equation}

The next formula for $\card{\bS_r(n, k)}$ is derived using ideas from~\cite{subsetmodprime,jpl}, even though it does not appear explicitly in these papers. A similar argument was also used in our proof of Lemma~\ref{lem:S1}.

\begin{lemma}\label{lem:S'r}
    For any $0\le k\le n$ and $0\le r< n$, 
    \[
        \card{\bS_r(n, k)} = \frac{1}{n} \sum_{d \mid \gcd(n, k)} {n/d \choose k/d} c_d(r) (-1)^{k/d + k}
    \]
\end{lemma}

\begin{proof}
    Consider the generating function 
    \begin{equation*}
        \bar{f}(x,t) = \prod_{j = 1}^n (1 + x t^j)
    \end{equation*}
    for subsets of $[n]$ with respect to their size and their sum. As in the proof of Lemma~\ref{lem:S1}, letting $\nroot$ be a primitive $n$th root of unity, we have
    \begin{equation*}
        \sum_{k=0}^n \card{\bS_r(n, k)} x^k = \frac{1}{n} \sum_{\ell = 0}^{n-1} \nroot^{-r\ell} \bar{f}(x,\nroot^{\ell}).
    \end{equation*}
    Following the steps in that proof, and letting $m=n/\gcd(n,\ell)$, we can write
    $$\bar{f}(x,\nroot^\ell)=\prod_{j=1}^n(1+x\nroot^{\ell j})
    =\left(\prod_{s=0}^{m-1}(1+x\nroot^{\ell s})\right)^{n/m}=\left(1-(-x)^m\right)^{n/m},$$
    and deduce
    \begin{equation*}
        \sum_{k=0}^n \card{\bS_r(n, k)} x^k = \frac{1}{n} \sum_{d \mid n} \left(1 - (-x)^{n/d}\right)^d \left( \sum_{\substack{\ell = 0 \\ \gcd(n, \ell) = d}}^{n-1} \nroot^{-r\ell} \right).
    \end{equation*}
    Because $\gcd(n, \ell) = d$, the root of unity $\nroot^{-\ell}$ has order $n / d$.
    Therefore the rightmost summation is the sum of the $r$th powers of roots of
    unity of order $n / d$, which is precisely $c_{n/d}(r)$. Replacing $d$ with $n/d$, we obtain
    \begin{equation*}
        \sum_{k=0}^n \card{\bS_r(n, k)} x^k = \frac{1}{n} \sum_{d \mid n} \left(1 - (-x)^d\right)^{n/d} c_d(r).
    \end{equation*}
    Finally, applying the binomial theorem to $\left(1 - (-x)^d\right)^{n/d}$ and extracting the coefficient of $x^k$ yields the stated expression for $\card{\bS_r(n, k)}$.
\end{proof}

On the other hand, a similar formula for $\card{N_r(n, k)}$ follows from recent work of Hadjicostas \cite[Thm.~1]{hadji-necklaces} enumerating $q$-ary necklaces with co-period dividing~$r$.

\begin{lemma}[\cite{hadji-necklaces}]\label{lem:Nr}
    For any $0\le k\le n$ and $0\le r< n$,
    \[
        \card{N_r(n, k)} = \frac{1}{n} \sum_{d \mid \gcd(n, k)} {n/d \choose k/d} c_d(r).
    \]
\end{lemma}

\subsection{Easy consequences}
The fact that the expressions in Lemmas~\ref{lem:S'r} and~\ref{lem:Nr} are so similar allows us to use number-theoretic properties to deduce some
combinatorial equalities.
Let us start with the simple case when $\gcd(n, k, r) = 1$. 

\begin{proposition}\label{prop:gcd=1}
    If $\gcd(n, k, r) = 1$, then $\card{\bS_r(n, k)} = \card{\bS_1(n, k)}$ and $\card{N_r(n, k)} = \card{N_1(n, k)}$. 
\end{proposition}

\begin{proof}
    We have $\gcd(r, d) = 1$ for all $d \mid \gcd(n, k)$. Thus, by equation~\eqref{eq:Ramanujan}, $c_d(r) = c_d(1)$ in this case. The equalities now follow from Lemmas~\ref{lem:S'r} and~\ref{lem:Nr}, respectively. 
\end{proof}

It is natural to ask whether Proposition~\ref{prop:gcd=1} has a bijective proof.
The second equality is straightforward: in fact, $N_r(n, k)= N_1(n, k)$ when $\gcd(n, k, r) = 1$, since the co-period of a necklace in $N_r(n, k)$ has to divide $r$, $n$, and $k$. On the other hand, proving the equality $\card{\bS_r(n, k)} = \card{\bS_1(n, k)}$ bijectively will take more work. The bijection will rely on the following lemma. 

\begin{lemma}\label{lem:Bezout}
    If $\gcd(n, k, r) = 1$, then there exist integers $x, y, z$ such that
    \begin{equation}\label{eq:Bezout}
        nx + ky + rz = 1
    \end{equation}
    and $\gcd(z, n) = 1$.
\end{lemma}

\begin{proof}
    The existence of $x, y, z$ satisfying equation~\eqref{eq:Bezout} is guaranteed by
    B\'ezout's identity. Given a solution $(x,y,z)$, any triple of the form $(x,
    y - rj, z + kj)$, where $j$ is an arbitrary integer, is a solution to
    equation~\eqref{eq:Bezout} as well. Therefore, it suffices to show there is
    some integer $j$ such that $\gcd(z + kj, n) = 1$.

    Let $p_1,\dots,p_m$ be the prime factors of $n$. We want to find a $j$ such that $p_i\nmid z+kj$ for all $i$. For factors such that $p_i\mid k$, having $p_i\mid z+kj$ would imply $p_i\mid z$. But then $p_i$ would divide $n$, $k$, and $z$, contradicting equation~\eqref{eq:Bezout}.
    
    For factors such that $p_i\nmid k$, the condition $p_i\nmid z+kj$ can be expressed as $j \not\equiv
    -z / k \pmod{p_i}$. There are finitely many such factors $p_i$, each one
    forcing $j$ to avoid a single congruence class modulo $p_i$. By the Chinese
    remainder theorem, there are infinitely many $j$ satisfying all these
    constraints.
\end{proof}

\begin{proposition}    
    If $\gcd(n, k, r) = 1$, then there exist integers $y$ and $z$ such that the
    map $A \mapsto (zA + y) \bmod n$ is a bijection between $\bS_r(n, k)$ and
    $\bS_1(n, k)$.
\end{proposition}

\begin{proof}
    By Lemma~\ref{lem:Bezout}, there exist integers $x, y, z$ such that
    $nx + ky + rz = 1$ and $\gcd(z, n) = 1$. 
    Given a set $A=\{a_i:1\le i\le k\}\in \bS_r(n, k)$, consider the set $A'=\{z a_i + y \bmod n:1\le i \le k\}$. Since $\gcd(z,n)=1$, the set $A'$ contains $k$ distinct elements, and the map $A\mapsto A'$ is invertible. Additionally, the sum of the elements of $A'$ is congruent to 
    $$\sum_{i=1}^k (z a_i + y) \equiv z \left(\sum_{i=1}^k a_i\right) + ky \equiv zr + ky = 1 - nx \equiv 1 \pmod{n}.\qedhere$$
\end{proof}

\begin{example}
    For $n=6$, $k=3$ and $r=2$, the integers $x=0$, $y=1$ and $z=-1$ satisfy $nx + ky + rz = 1$ and $\gcd(z, n) = 1$. Thus, the map $A \mapsto (-A + 1) \bmod n$ is a bijection from $\bS_2(6,3)=\{\{0,3,5\},\{1,2,5\},\{1,3,4\}\}$ to $\bS_1(6,3)=\{\{1,2,4\},\{0,2,5\},\{0,3,4\}\}$.
\end{example}

Another interesting case occurs when $n$ or $k$ are odd. By comparing Lemmas~\ref{lem:S'r} and~\ref{lem:Nr}, we see that the only difference in the formulas for $\card{\bS_r(n, k)}$ and $\card{N_r(n, k)}$ is the alternating factor $(-1)^{k/d + k}$. 
If either $n$ or $k$ are odd, then $(-1)^{k/d + k} = 1$ for all $d \mid \gcd(n, k)$, and we obtain the following result.

\begin{theorem}\label{thm:odd_nk}
    If $n$ or $k$ are odd, then $$\card{\bS_r(n, k)} = \card{N_r(n,
    k)}.$$ Consequently, $\card{\bS_r(n)} = \card{N_r(n)}$ if $n$ is odd.
\end{theorem}

This generalizes the fact that $\card{\bS_0(n, k)} = \card{N(n, k)}$ for odd $n$, which was mentioned in Section~\ref{sec:r=0}, and the fact that 
$\card{\bS_1(n, k)} = \card{L(n, k)}$ for odd $n$ or $k$, given in Corollary~\ref{cor:S'1=L}. 

\subsection{Characterizing equality}

The converse of Theorem~\ref{thm:odd_nk}, as stated, does not hold:
there are cases where $n$ and $k$ are both even, yet $\card{\bS_r(n, k)} = \card{N_r(n, k)}$ (see Table~\ref{tab:diff}).
Our next goal is to provide necessary and sufficient conditions on $n,k,r$ for this equality to hold, and to determine by how much it fails otherwise. Next we state the main result in this section.
We use $\nu_2(m)$ to denote the 2-adic valuation function, defined as the highest power of $2$ that divides $m$. In particular, $\nu_2(0)=\infty$.

\begin{theorem}\label{thm:Nr=S'r}
   Let $0\le k\le n$ and $0\le r< n$. Then $\card{\bS_r(n, k)} = \card{N_r(n, k)}$ if and only if any of the following conditions hold:
    \begin{enumerate}[(a)]
        \item $n$ is odd,
        \item $k$ is odd,
        \item $\nu_2(n) < \nu_2(k)$ (including $k = 0$),
        \item $\nu_2(k) - \nu_2(r) \geq 2$,
        \item $k = n$ is even and $r \notin \{0, n / 2\}$.
    \end{enumerate}
    Additionally, when none of these conditions hold, we have 
    \begin{equation}\label{eq:difference}
        \card{N_r(n, k)} - \card{\bS_r(n, k)}= \pm \card{N_r(n/2^{\nu_2(k)}, k/2^{\nu_2(k)})}
        =
        \pm \card{\bS_r(n/2^{\nu_2(k)}, k/2^{\nu_2(k)})},
    \end{equation}
    where the signs are negative if $\nu_2(k) - \nu_2(r) =1$, and positive otherwise.
\end{theorem}

\begin{example}
    For $r=2$, $k=6$, and $n=2m$ (where $m\ge3$), Theorem~\ref{thm:Nr=S'r} states that $\card{\bS_2(2m, 6)} = \card{N_2(2m, 6)}$ if and only if $m=3$. For $m>3$, we have
    \begin{equation*}
        \card{N_2(2m, 6)} - \card{\bS_2(2m, 6)} = \card{N_2(m, 3)}.
    \end{equation*}
    As shown in Table~\ref{tab:diff}, the first few terms of this sequence, starting at $m=4$, are $
        1, 2, 3, 5, 7, 9, 12, 15, 18, \dots$.
    This is sequence A001840 in~\cite{oeis}, which also counts primitive necklaces of length $m$ with $3$ ones, noting that $N_2(n, 3) = N_1(n, 3)$.
\end{example}

\setlength{\tabcolsep}{4pt}
\begin{table}[h]
\resizebox{\textwidth}{!}{
\begin{tabular}{r|cccccccccccccccccccc}
$m\backslash k$ &0&1&2&3&4&5&6&7&8&9&10&11&12&13&14&15&16&17&18&19 \\ \hline
1 & 0 & 1&&&&&&&&&&&&&&&&&& \\
2 & 0 & 1 & -1&&&&&&&&&&&&&&&&& \\
3 & 0 & 1 & 0 & 0&&&&&&&&&&&&&&&& \\
4 & 0 & 1 & -1 & 1 & 0&&&&&&&&&&&&&&& \\
5 & 0 & 1 & 0 & 2 & 0 & 0&&&&&&&&&&&&&& \\
6 & 0 & 1 & -1 & 3 & 0 & 1 & 0&&&&&&&&&&&&& \\
7 & 0 & 1 & 0 & 5 & 0 & 3 & 0 & 0&&&&&&&&&&&& \\
8 & 0 & 1 & -1 & 7 & 0 & 7 & -1 & 1 & 0&&&&&&&&&&& \\
9 & 0 & 1 & 0 & 9 & 0 & 14 & 0 & 4 & 0 & 0&&&&&&&&&& \\
10 & 0 & 1 & -1 & 12 & 0 & 25 & -2 & 12 & 0 & 1 & 0&&&&&&&&& \\
11 & 0 & 1 & 0 & 15 & 0 & 42 & 0 & 30 & 0 & 5 & 0 & 0&&&&&&&& \\
12 & 0 & 1 & -1 & 18 & 0 & 66 & -3 & 66 & 0 & 18 & -1 & 1 & 0&&&&&&& \\
13 & 0 & 1 & 0 & 22 & 0 & 99 & 0 & 132 & 0 & 55 & 0 & 6 & 0 & 0&&&&&& \\
14 & 0 & 1 & -1 & 26 & 0 & 143 & -5 & 245 & 0 & 143 & -3 & 26 & 0 & 1 & 0&&&&& \\
15 & 0 & 1 & 0 & 30 & 0 & 200 & 0 & 429 & 0 & 333 & 0 & 91 & 0 & 7 & 0 & 0&&&& \\
16 & 0 & 1 & -1 & 35 & 0 & 273 & -7 & 715 & 0 & 715 & -7 & 273 & 0 & 35 & -1 & 1 & 0&&& \\
17 & 0 & 1 & 0 & 40 & 0 & 364 & 0 & 1144 & 0 & 1430 & 0 & 728 & 0 & 140 & 0 & 8 & 0 & 0&& \\
18 & 0 & 1 & -1 & 45 & 0 & 476 & -9 & 1768 & 0 & 2700 & -14 & 1768 & 0 & 476 & -4 & 45 & 0 & 1 & 0& \\
19 & 0 & 1 & 0 & 51 & 0 & 612 & 0 & 2652 & 0 & 4862 & 0 & 3978 & 0 & 1428 & 0 & 204 & 0 & 9 & 0 & 0 \\
\end{tabular}}
    \caption{The differences $\card{N_2(2m, 2k)} - \card{\bS_2(2m, 2k)}$ for $1 \leq m \le 19$ and $0\le k\le m$.}
    \label{tab:diff}
\end{table}

The proof of Theorem~\ref{thm:Nr=S'r} will rely on the following lemma.

\begin{lemma}
    \label{lem:trivial-conditions}
    If any of conditions (a), (b), (c) or (d) in Theorem~\ref{thm:Nr=S'r} hold, then $\card{\bS_r(n, k)} = \card{N_r(n, k)}$.
    If none of these hold, then $\card{N_r(n, k)} - \card{\bS_r(n, k)}$ is given by equation~\eqref{eq:difference}.
\end{lemma}

\begin{proof}
    By Lemmas~\ref{lem:S'r} and~\ref{lem:Nr}, we have
    \begin{align} \nonumber
    n\left(\card{N_r(n,k)}-\card{\bS_r(n,k)}\right)&=
        \sum_{d \mid \gcd(n, k)} {n/d \choose k/d} c_d(r)\,\left(1 - (-1)^{k/d + k}\right)\\
        \label{eq:difference2}
        &=
        2\sum_{d \mid \gcd(n, k)} {n/d \choose k/d} c_d(r)\,[k / d + k \text{ odd}].
    \end{align}
    We will show that if any of conditions (a), (b), (c) and (d) hold, then all the terms in this sum vanish.

    Every term where $d$ is odd vanishes because $k/d$ and $k$ have the same parity. This is the case when either $n$ or $k$ are odd, that is, when conditions (a) or (b) hold.

    Suppose that both $n$ and $k$ are even. Then $k/d +
    k$ is odd if and only if $k/d$ is odd, which happens precisely when
    $\nu_2(d) = \nu_2(k)$. In particular, if condition (c) holds, then $\nu_2(d)
    \leq \nu_2(n) < \nu_2(k)$ for any $d$ dividing $\gcd(n,k)$, so all the summands in equation~\eqref{eq:difference2} are zero.

    Suppose now that $\nu_2(n) \geq \nu_2(k)$. Letting $m = \nu_2(k)$, we
    can write 
    \begin{equation}\label{eq:k'n'}
    k=2^m k', \quad n=2^m n', \quad  d=2^m d',
    \end{equation}
    where $k'$ is odd. Equation~\eqref{eq:difference2} becomes
    \begin{equation}\label{eq:difference3}
        2 \sum_{d' \mid \gcd(n', k')} {n'/d' \choose k'/d'} c_{2^m d'}(r)=2\, c_{2^m}(r) \sum_{d' \mid \gcd(n', k')} {n'/d' \choose k'/d'} c_{d'}(r),
    \end{equation}
    where we used the fact that $c_{ab}(r)=c_a(r)c_b(r)$ if $\gcd(a,b)=1$. 
    
    The expression for $c_{d}(r)$ from the right-hand side of equation~\eqref{eq:Ramanujan} shows that $c_{d}(r)=0$ if and only if  $\mu(d / \gcd(r, d))=0$, which is equivalent to $d / \gcd(r, d)$ not being squarefree. In particular, $c_{2^m}(r)=0$ if and only if $m - \nu_2(r) \geq 2$. Thus, when condition (d) holds, equation~\eqref{eq:difference3} vanishes as well.

    The above argument shows that $\card{\bS_r(n, k)} = \card{N_r(n, k)}$ if any of conditions (a), (b), (c) or (d) hold.
    Suppose now that none of these conditions hold, that is, $\nu_2(n) \geq m\ge1$ and $m - \nu_2(r) < 2$. By Lemma~\ref{lem:Nr}, the sum in the right-hand side of equation~\eqref{eq:difference3} is precisely $n' \card{N_r(n', k')}$, which also equals
    $n' \card{\bS_r(n', k')}$ by Theorem~\ref{thm:odd_nk}, since $k'$ is odd. It follows that
    \begin{equation}
        \label{eq:penultimate}
        \card{N_r(n,k)}-\card{\bS_r(n,k)}=\frac{c_{2^m}(r)}{2^{m-1}} \card{N_r(n', k')}
        = \frac{c_{2^m}(r)}{2^{m-1}} \card{\bS_r(n', k')}.
    \end{equation}

    Finally, we can directly compute $c_{2^m}(r)$ using the closed-form expression in equation~\eqref{eq:Ramanujan}, noting that $$\gcd(2^m,r)=2^{\min(m,\nu_2(r))}=\begin{cases} 2^{m-1} & \text{if $\nu_2(r)=m-1$}, \\ 2^m & \text{otherwise}.
    \end{cases}$$
    Thus,
    \begin{align*}
        \frac{c_{2^m}(r)}{2^{m-1}} &= \frac{\mu \left( \frac{2^m}{\gcd(2^m, r)} \right) }{\phi \left( \frac{2^m}{\gcd(2^m, r)}\right)} \frac{\phi(2^m)}{2^{m-1}} = \begin{cases} \mu(2)/\phi(2) = -1 & \text{if $\nu_2(r)=m-1$}, \\ \mu(1)/\phi(1) = 1  & \text{otherwise}.
        \end{cases}
    \end{align*}
    Substituting in~\eqref{eq:penultimate}, we obtain the
    expression in equation~\eqref{eq:difference}.  
\end{proof}

\begin{proof}[Proof of Theorem~\ref{thm:Nr=S'r}] 
We know by Lemma~\ref{lem:trivial-conditions} that if any of conditions (a), (b), (c) or (d) hold, then $\card{\bS_r(n, k)} = \card{N_r(n, k)}$.
Let us assume that none of these conditions hold. In this case, Lemma~\ref{lem:trivial-conditions} tells us that
\begin{equation*}
    \card{N_r(n, k)} - \card{\bS_r(n, k)} =  \pm\card{N_r(n', k')},
\end{equation*}
with $n'$ and $k'$ given by equation~\eqref{eq:k'n'}, so it suffices to determine when this quantity is equal to zero.

If $0 < k < n$, then $0<k'<n'$, and
$\card{N_r(n', k')}\neq 0$ because the set $N_r(n', k')$ contains the primitive necklace consisting of $n'-k'$ zeros followed by $k'$ ones.

For $k = n$, the only necklace of length $n$ with $n$ ones has co-period $n$, and the only subset of $[n]$ with $n$ elements sums to $n(n + 1) / 2 \equiv n(n - 1) / 2 \equiv n/2 \pmod{n}$, using the assumption that $n$ is even since (a) does not hold. 
Therefore,
\begin{equation*}\label{eq:k=n}
    \card{N_r(n, n)} =
    \begin{cases}
        1 & \text{if } r = 0, \\
        0 & \text{otherwise};
    \end{cases} \qquad
    \card{\bS_r(n, n)} =
    \begin{cases}
        1 & \text{if } r = n/2, \\
        0 & \text{otherwise}.
    \end{cases}
\end{equation*}
These quantities are equal precisely when $r\notin\{0,n/2\}$ (condition (e)).
\end{proof}

\subsection{Special cases}
It was already observed by~\cite{Chan} that $\card{\bS_0(n, k)}=\card{N(n, k)}$ for any odd $n$ and any $0 \leq k \leq n$ (see Problem~\ref{prob:Chan}). As an
application of Theorem~\ref{thm:Nr=S'r} in the case $r=0$, we can determine precisely when this equality holds. 

\begin{corollary}
    \label{cor:r=0-gen}
    For $0\le k\le n$, we have $\card{\bS_0(n, k)} = \card{N(n, k)}$ if and only if $n$ is odd, $k$ is odd, or $\nu_2(n) < \nu_2(k)$ (including $k = 0$).
    Additionally, if $\nu_2(n) \geq \nu_2(k) > 0$, then
    \begin{equation*}
        \card{N(n, k)} - \card{\bS_0(n, k)} = \card{N_0(n / 2^{\nu_2(k)}, k / 2^{\nu_2(k)})} = \card{\bS_0(n / 2^{\nu_2(k)}, k / 2^{\nu_2(k)})}.
    \end{equation*}
\end{corollary}

We can rephrase Corollary~\ref{cor:r=0-gen} by saying that 
 $\card{S_0^+(n, k)}=\card{N(n, k)}$ for all $0\le k\le n$, where 
\begin{equation*}
    S_0^+(n, k)
    =
    \begin{cases}
        \bS_0(n, k) & \text{if $n$ or $k$ are odd, or $\nu_2(n) < \nu_2(k)$,} \\
        \bS_0(n, k) \cup \bS_0(n / 2^{\nu_2(k)}, k / 2^{\nu_2(k)}) & \text{otherwise}.
    \end{cases}
\end{equation*}

Similarly, in the case $r=1$, Theorem~\ref{thm:Nr=S'r} 
gives the following, where the conditions $\nu_2(k)\ge2$ and $\nu_2(k)=1$ have been stated as $4\mid k$ and $k\equiv 2\bmod 4$, respectively.

\begin{corollary}
    \label{cor:r=1-gen}
    For $0\le k\le n$, we have $\card{\bS_1(n, k)} = \card{L(n, k)}$ if and only if any of the following conditions hold:
    \begin{enumerate}[(a)]
        \item $n$ is odd,
        \item $k$ is odd,
        \item $\nu_2(n) < \nu_2(k)$ (including $k = 0$),
        \item $4\mid k$,
        \item $k = n$ is even and greater than $2$.
    \end{enumerate}
    Additionally, if $n$ is even and $k\equiv 2\bmod 4$, then 
    $\card{L(n, k)} - \card{\bS_1(n, k)} = -\card{L(n/2, k/2)}$.
\end{corollary}

In other words, we have $\card{\bS_1(n, k)} = \card{L^+(n, k)}$ for all $0\le k\le n$, recovering Corollary~\ref{cor:S'1=L},

By summing the formulas in Lemmas~\ref{lem:S'r} and~\ref{lem:Nr} over all $k$, we obtain
    \begin{align*}
        \card{S_r(n)} &= \frac{1}{n} \sum_{d \mid n} (1 - (-1)^d)^{n/d} c_d(r),\\
        \card{N_r(n)} &= \frac{1}{n} \sum_{d \mid n} 2^{n / d} c_d(r).
    \end{align*}
Taking the difference, we get
\begin{equation}\label{eq:Nr-S}
        \card{N_r(n)} - \card{\bS(n)}
            = \frac{1}{n} \sum_{\substack{d \mid n \\ d \text{ even}}} 2^{n/d} c_d(r),
\end{equation}
which vanishes if $n$ is odd, as we noted in Theorem~\ref{thm:odd_nk}.

A consequence of equation~\eqref{eq:Nr-S} is that, for any given $n$, the difference $\card{N_r(n)} - \card{\bS_r(n)}$ can only take as many distinct values as the number of divisors of $n$. This is because the only appearance of $r$ is in the factor $c_d(r)$, which, by equation~\eqref{eq:Ramanujan}, is determined by $\gcd(r,n)$. See Table~\ref{tab:diff-sum} for an example.

\begin{table}[h]
\resizebox{\textwidth}{!}{
\begin{tabular}{r|rrrrrrrrrrrrrrrrrrrrr}
$n \backslash r$ & 0 & 1 & 2 & 3 & 4 & 5 & 6 & 7 & 8 & 9 & 10 & 11 & 12 & 13 & 14 & 15 & 16 & 17 & 18 & 19 & 20 \\
\hline
1 & 0 & 0 &  &  &  &  &  &  &  &  &  &  &  &  &  &  &  &  &  &  &  \\
2 & 1 & -1 & 1 &  &  &  &  &  &  &  &  &  &  &  &  &  &  &  &  &  &  \\
3 & 0 & 0 & 0 & 0 &  &  &  &  &  &  &  &  &  &  &  &  &  &  &  &  &  \\
4 & 2 & -1 & 0 & -1 & 2 &  &  &  &  &  &  &  &  &  &  &  &  &  &  &  &  \\
5 & 0 & 0 & 0 & 0 & 0 & 0 &  &  &  &  &  &  &  &  &  &  &  &  &  &  &  \\
6 & 2 & -1 & 1 & -2 & 1 & -1 & 2 &  &  &  &  &  &  &  &  &  &  &  &  &  &  \\
7 & 0 & 0 & 0 & 0 & 0 & 0 & 0 & 0 &  &  &  &  &  &  &  &  &  &  &  &  &  \\
8 & 4 & -2 & 1 & -2 & 2 & -2 & 1 & -2 & 4 &  &  &  &  &  &  &  &  &  &  &  &  \\
9 & 0 & 0 & 0 & 0 & 0 & 0 & 0 & 0 & 0 & 0 &  &  &  &  &  &  &  &  &  &  &  \\
10 & 4 & -3 & 3 & -3 & 3 & -4 & 3 & -3 & 3 & -3 & 4 &  &  &  &  &  &  &  &  &  &  \\
11 & 0 & 0 & 0 & 0 & 0 & 0 & 0 & 0 & 0 & 0 & 0 & 0 &  &  &  &  &  &  &  &  &  \\
12 & 8 & -5 & 4 & -6 & 6 & -5 & 4 & -5 & 6 & -6 & 4 & -5 & 8 &  &  &  &  &  &  &  &  \\
13 & 0 & 0 & 0 & 0 & 0 & 0 & 0 & 0 & 0 & 0 & 0 & 0 & 0 & 0 &  &  &  &  &  &  &  \\
14 & 10 & -9 & 9 & -9 & 9 & -9 & 9 & -10 & 9 & -9 & 9 & -9 & 9 & -9 & 10 &  &  &  &  &  &  \\
15 & 0 & 0 & 0 & 0 & 0 & 0 & 0 & 0 & 0 & 0 & 0 & 0 & 0 & 0 & 0 & 0 &  &  &  &  &  \\
16 & 20 & -16 & 14 & -16 & 17 & -16 & 14 & -16 & 18 & -16 & 14 & -16 & 17 & -16 & 14 & -16 & 20 &  &  &  &  \\
17 & 0 & 0 & 0 & 0 & 0 & 0 & 0 & 0 & 0 & 0 & 0 & 0 & 0 & 0 & 0 & 0 & 0 & 0 &  &  &  \\
18 & 30 & -28 & 28 & -29 & 28 & -28 & 29 & -28 & 28 & -30 & 28 & -28 & 29 & -28 & 28 & -29 & 28 & -28 & 30 &  &  \\
19 & 0 & 0 & 0 & 0 & 0 & 0 & 0 & 0 & 0 & 0 & 0 & 0 & 0 & 0 & 0 & 0 & 0 & 0 & 0 & 0 &  \\
20 & 56 & -51 & 48 & -51 & 54 & -52 & 48 & -51 & 54 & -51 & 48 & -51 & 54 & -51 & 48 & -52 & 54 & -51 & 48 & -51 & 56
\end{tabular}}
\caption{The differences $\card{N_r(n)} - \card{\bS_r(n)}$ for $1 \leq n \leq 20$ and $0 \leq r \leq n$.}
    \label{tab:diff-sum}
\end{table}

\section{Generalizations to $q$-ary necklaces}\label{sec:q}

This paper has focused on binary necklaces and subsets, following the approach from \cite{Baker,subsetmodprime,jpl}. However, for any integer $q\ge2$, one can define necklaces over the alphabet $\{0,1,\dots,q-1\}$, sometimes called {\em $q$-ary} or {\em $q$-colored necklaces}. In~\cite{Chan}, Chan shows that, if $\gcd(q,n)=1$, then the number of $q$-ary necklaces of length $n$ equals the number of multisets of $[n]$ whose sum is divisible by $n$ and where the multiplicity of each element is less than $q$. 
This generalizes the fact that $\card{\bS_0(n)}=\card{N(n)}$ for odd $n$, as discussed in Section~\ref{sec:r=0}. 
In fact, Chan gives a bijective proof of his equality when $q$ is a prime power. This bijection has recently been generalized by Li and Zhou~\cite{LZ} to any $q$ such that $\gcd(q,n)=1$.

In this section we briefly discuss how some of our generalized results involving the parameters $r$ and $k$ can be extended from binary necklaces to $q$-ary necklaces. For this extension to work, the interpretation of the parameter $k$ that we use on $q$-ary necklaces is the sum of the entries, viewed as integers between $0$ and $q-1$. 

\begin{definition} 
For any integers $0\le k\le n$, $0\le r< n$, and $q\ge2$, define the following sets:
\begin{itemize}
    \item $\bS_r^q(n,k)$ is the set of $k$-element multisets of $[n]$ that sum to $r\bmod n$ where each element has multiplicity less than $q$.
    \item $N_r^q(n,k)$ is the set of necklaces over $\{0,1,\dots,q-1\}$ whose entries sum to $k$ and whose co-period divides $r$.
\end{itemize}
\end{definition}

Next we generalize Lemmas~\ref{lem:S'r} and~\ref{lem:Nr} to arbitrary $q$. It is more convenient to give expressions for the generating functions where we sum over~$k$.

\begin{theorem}
    \label{thm:q-S}
    For any $0\le r< n$ and any $q \ge 2$,
    \begin{equation*}
        \sum_{k=0}^n \card{\bS_r^q(n, k)} x^k
        =
        \frac{1}{n}
        \sum_{d \mid n} c_{d}(r)
            \left(
                \frac{\left(1 - x^{\frac{qd}{\gcd(q, d)}}\right)^{\gcd(q, d)}}{1 - x^d}
            \right)^{n/d}.
    \end{equation*}
\end{theorem}

\begin{proof}
Generalizing the proof of Lemma~\ref{lem:S'r}, consider now the generating function 
    \begin{equation*}
        \bar{f}_q(x,t) = \prod_{j = 1}^n (1 + x t^j+x^2t^{2j}+\dots+x^{q-1}t^{(q-1)j}) = \prod_{j = 1}^n \frac{1-x^qt^{qj}}{1-zt^j}
    \end{equation*}
    for multisets of $[n]$ where each element has multiplicity less than $q$, with respect to their size and their sum. Letting $\nroot$ be a primitive $n$th root of unity, we have
    \begin{equation}
        \label{eqn:big-sum}
        \sum_{k=0}^n \card{\bS_r^q(n, k)} x^k
        =
        \frac{1}{n} \sum_{\ell = 0}^{n-1} \nroot^{-r\ell} \bar{f}_q(x,\nroot^{\ell}).
    \end{equation}
    For any given $\ell$, we can evaluate $\bar{f}_q(x, \nroot^\ell)$ directly:
    \[
        \bar{f}_q(x, \nroot^\ell)
        = \prod_{j = 1}^n (1-x^q\nroot^{\ell qj}) \prod_{j = 1}^n (1-z\nroot^{\ell j})^{-1}.
    \]
    For the first product, if we set $m \coloneqq n / \gcd(n, \ell q)$, then $w \coloneqq
    \nroot^{\ell q}$ is a root of unity of order $m$, so the product is
    \[
             \prod_{j = 1}^n (1 - x^q w^j)
            = \left(\prod_{s = 0}^{m-1} (1 - x^q w^s)\right)^{n/m}
            =\left(1 - x^{qm}\right)^{n/m},
    \]
    as in the proof of Lemma~\ref{lem:S'r}.

    The second product is handled similarly, with $m' \coloneqq n / \gcd(n, \ell)$ and $w' \coloneqq
    \nroot^{\ell}$ playing the roles of $m$ and $w$. We obtain
    \[
        \bar{f}_q(x, \nroot^\ell)
        = \frac{\left(1 - x^{qm}\right)^{n/m}}{\left(1 - x^{m'}\right)^{n/m'}}
        = \frac{ \left(1 - x^{\frac{qn}{\gcd(n, \ell q)}}\right)^{\gcd(n, \ell q)}}{\left(1 - x^{\frac{n}{\gcd(n, \ell)}}\right)^{\gcd(n, \ell)}}.
    \]
    Letting $d = \gcd(n, \ell)$, and noting that $
        \gcd(n, \ell q) = d\, \gcd \left( n/d, q\right)$,
    we can write
    \begin{equation*}
        \bar{f}_q(x, \nroot^\ell)
        = \left(\frac{\left(1 - x^{\frac{qn/d}{\gcd(n/d, q)}}\right)^{\gcd(n/d, q)}}{1 - x^{n / d}} \right)^d.
    \end{equation*}
    Plugging this into equation~\eqref{eqn:big-sum} and grouping according to $d = \gcd(n, \ell)$, we obtain
    \begin{equation*}
        \sum_{k=0}^n \card{\bS_r^q(n, k)} x^k
        =
        \frac{1}{n} \sum_{d \mid n} \left(\frac{\left(1 - x^{\frac{qn/d}{\gcd(n/d, q)}}\right)^{\gcd(n/d, q)}}{1 - x^{n / d}} \right)^d \left( \sum_{\substack{\ell = 0 \\ \gcd(n, \ell) = d}}^{n-1}  \nroot^{-r\ell} \right).
    \end{equation*}
    As before, the inner sum equals $c_{n/d}(r)$. Finally, if we replace $d$ with $n / d$, we obtain the stated expression.
\end{proof}

\begin{theorem}
    \label{thm:q-N}
    For any $0\le r< n$ and any $q \ge 2$,
    \begin{equation*}
    \sum_{k=0}^n \card{N_r^q(n, k)} x^k
        = \frac{1}{n} \sum_{d \mid n} c_{d}(r)
        \left(\frac{1 - x^{qd}}{1 - x^d}\right)^{n/d}
    \end{equation*}
\end{theorem}

\begin{proof}
    For any $q$-tuple of nonnegative integers $\mathbf{n} = (n_0, n_1, \dots,
    n_{q - 1})$, let $M_r(\mathbf{n})$ be the number of $q$-ary necklaces of
    length $|\mathbf{n}| = n_0 + n_1 + \cdots + n_{q - 1}$
    with $n_i$ copies of $i$ for every $i$, and whose co-period divides
    $r$. This is related to our generating function via the equation
    \begin{equation}\label{eq:NM}
        \sum_{k=0}^n \card{N_r^q(n, k)} x^k = \sum_{|\mathbf{n}| = n} M_r(\mathbf{n})\, x^{n_1 + 2 n_2 + \cdots + (q - 1) n_{q - 1}}.
    \end{equation}

    In \cite[Thm.~1, eq.~(18)]{hadji-necklaces}, 
    Hadjicostas shows that
    \begin{equation*}
        \sum_{\mathbf{n} \neq \mathbf{0}} M_r(\mathbf{n}) \prod_{i = 0}^{q - 1} x_i^{n_i} = - \sum_{d \geq 1} \frac{c_d(r)}{d} \log(1 - x_0^d - x_1^d - \cdots - x_{q - 1}^d).
    \end{equation*}
    Setting $x_i = yx^i$ for all $i$, this equation becomes
    \[
        \sum_{\mathbf{n} \neq \mathbf{0}} M_r(\mathbf{n})\, y^{|\mathbf{n}|} x^{n_1 + 2 n_2 + \cdots + (q - 1) n_{q - 1}}
        = - \sum_{d \geq 1} \frac{c_d(r)}{d} \log\left(1 - y^d \frac{1 - x^{qd}}{1 - x^d}\right).
    \]
    Extracting the coefficient of $y^n$ on both sides, using equation~\eqref{eq:NM} and the series expansion $-\log(1-z)=\sum_{j\ge1}z^j/j$, we obtain
    \[
        \sum_{k=0}^n \card{N_r^q(n, k)} x^k 
        =
        \sum_{d\mid n} \frac{c_d(r)}{d} \frac{1}{n/d} \left(\frac{1 - x^{qd}}{1 - x^d}\right)^{n/d}
        =
        \frac{1}{n} \sum_{d\mid n} c_d(r) \left(\frac{1 - x^{qd}}{1 - x^d}\right)^{n/d}.\qedhere
    \]
\end{proof}

The following result generalizes Theorem~\ref{thm:odd_nk} in the case of odd $n$.

\begin{corollary}\label{cor:q-ary}
    If $\gcd(q,n)=1$, then
    $$\card{\bS_r^q(n, k)} = \card{N_r^q(n,k)}.$$ 
\end{corollary}

\begin{proof}
    If $\gcd(q, n) = 1$, then $\gcd(q, d) = 1$ for all $d \mid n$, and the formulas in Theorems \ref{thm:q-S} and
    \ref{thm:q-N} are identical in this case.
\end{proof}

\section{Open questions}\label{sec:open}

As is the case for many of the results which inspired this work, finding bijective proofs of our results seems challenging. We end by listing some specific open questions. The first asks for a bijective proof of Theorem~\ref{thm:S=CVP}.

\begin{problem}
    Find a bijection between $S_1(n,k-1)$ and $\CVP(n,k)$ for $1\le k\le n$.
\end{problem}

Using the conditions in Corollary~\ref{cor:r=0-gen}, we can pose the following slight generalization of Problem~\ref{prob:Chan}.

\begin{problem}\label{prob:r=0}
    Find a bijection between $\bS_0(n, k)$ and $N(n, k)$ when $n$ or $k$ are odd, or $\nu_2(n)<\nu_2(k)$.
\end{problem}

The next question asks for a bijective proof of Corollary~\ref{cor:S'1=L}.

\begin{problem}\label{prob:r=1}
    Find a bijection between $\bS_1(n,k)$ and $L^+(n,k)$ for $n\ge2$ and $0\le k\le n$. Interesting special cases include a bijection between $\bS_1(n,k)$ and $L(n,k)$ when any of the conditions (a)--(e) in Corollary~\ref{cor:r=1-gen} hold, or a bijection between $\bS_1(n)$ and $L(n)$ when $n$ is odd.
\end{problem}

Problems~\ref{prob:r=0} and~\ref{prob:r=1} can be generalized to arbitrary $r$ as follows.

\begin{problem}
    Find a bijection between $\bS_r(n,k)$ and $N_r(n,k)$ for $0\le k\le n$ and
    $0\le r<n$ satisfying any of the conditions (a)--(e) in Theorem~\ref{thm:Nr=S'r}.
\end{problem}

When $r=1$, the diagram in Figure~\ref{fig:diagram} relates subsets, necklaces and permutations. By Corollary~\ref{cor:r=0-gen}, an analogue for $r = 0$ would have $S_0^+(n, k)$ and $N(n,k)$ playing the roles of $S'_1(n,k)$ and $L^+(n,k)$, respectively.

\begin{problem}
    Are there analogues of the diagram in Figure~\ref{fig:diagram} for arbitrary $r$? Specifically, are there sets of permutations playing the
    role of $\CVP(n, k)$?
\end{problem}

Our last problem concerns the generalization of Theorem~\ref{thm:Nr=S'r} to $q$-ary necklaces.

\begin{problem}\label{prob:q}
    Describe necessary and sufficient conditions on $n,k,r,q$ for which the equality $\card{\bS_r^q(n,k)}=\card{N_r^q(n,k)}$ holds.
\end{problem}

Corollary~\ref{cor:q-ary} gives the sufficient condition $\gcd(q,n)=1$.

\subsection*{Acknowledgments} 

The authors thank Matt Baker and Matthew Qian for bringing their conjecture to
our attention. The authors also thank Natasha Ter-Saakov for helpful
suggestions. Elizalde was partially supported by Simons Collaboration Grant
\#929653.

\bibliographystyle{plain}
\bibliography{subsets_necklaces}

\end{document}